\documentclass[journal ]{new-aiaa}
\usepackage[utf8]{inputenc}
\usepackage{textcomp}

\usepackage{graphicx}
\usepackage{amsmath}
\usepackage[version=4]{mhchem}
\usepackage{siunitx}
\usepackage{longtable,tabularx}
\usepackage{graphicx}
\usepackage{xcolor,comment}
\usepackage{float}
\usepackage{color}
\usepackage{graphicx}          
\usepackage[ruled,vlined]{algorithm2e}

\newcommand\bo[1]{\boldsymbol{#1}}

\newtheorem{theorem}{Theorem}

\newtheorem{remm}{Remark}

\newenvironment{proof}{\noindent {\em Proof.}}{\hfill \hspace*{1pt}\hfill $\square$\\}

\usepackage{algorithmicx}
\usepackage{algorithmicx}
\usepackage{optidef}

\title{Command Governors with Inexact Optimization and without Invariance}

\author{Emanuele Garone \footnote{Associate Professor, Service d'Automatique et Analyse des Systemes, AIAA Senior Member.}}
\affil{Universite Libre de Bruxelles, Brussels, Belgium}
\author{Ilya V. Kolmanovsky \footnote{Professor, Aerospace Engineering, AIAA Associate Fellow.}}
\affil{University of Michigan, Ann Arbour, MI}

\begin{document}

\maketitle



\section{Introduction}

\lettrine{R}{eference} Governors (RGs) \cite{garone2017reference} are add-on schemes to nominal closed-loop systems used
to enforce pointwise-in-time state and control constraints. An RG acts as a safety supervisor for the reference commands (set-points) given to the closed-loop system by a human 
operator or by a higher-level planner in autonomous vehicles.  The RG  monitors  reference commands to which the controller responds and modifies them if they create a danger of constraint violation to preserve safety. 

RGs are an attractive option for practitioners as it enables them to use a variety of techniques for control system design that do not explicitly handle constraints while relying on RG for constraint enforcement.  
Reference \cite{garone2017reference}
surveys the literature on aerospace and other proposed applications of RGs, the RG theory
and compares RG with alternative approaches to constraint handling such as Model Predictive Control \cite{borrelli2017predictive}; we do not replicate this survey here due to limited space.

A Command Governor (CG), first proposed in \cite{bemporad1997nonlinear} for linear discrete-time systems,
is a particular type of RG that modifies the reference command by finding a minimum norm projection of the original reference command 
onto a safe set of commands for the given state, i.e., a cross section of the safe set of state-command pairs.   This safe set is constructed as an invariant subset of the maximum output admissible set (MOAS) \cite{gilbert1991linear}, i.e., the set of all initial states and constant commands that yield response satisfying the imposed constraints.  When this safe set is polyhedral (this is the case, e.g., if CG is  designed based on a linear discrete-time model and state and control constraints are linear), the minimum norm projection is computed by solving a quadratic programming (QP) problem at each discrete-time instant.  Even though this QP problem is low dimensional, it typically has a large number of constraints due to the typically large number of affine inequalities needed to define the safe set.  Consequently, solving this QP problem online is challenging. Extensions of CG to nonlinear systems have been presented in \cite{bemporad1998reference}. 

In this Note we present a simple but very meaningful modification of the CG which enables it to operate with non-invariant safe sets and ensures feasibility and convergence properties even in the case the optimization is inexact. This modification significantly extends the applicability of the CG to practical problems where the construction of accurate invariant approximations of the MOAS and/or exact optimization may not be feasible due to model complexity or limited available onboard computing power.  We will illustrate the impact of this modification on an F-16 aircraft longitudinal flight control example in terms of computational time and memory reduction.


The paper is organized as follows. 
In Section~\ref{sec:1.5} we formally introduce the CG and further explain the significance of our contribution. 
In Section~\ref{sec:2} we highlight
the mechanism by which convergence of the modified by CG reference command to the original reference command is achieved in the existing CG theory.This informs the modification to CG presented in Section~\ref{sec:3} for which
we prove similar convergence results.  A simulation example of
longitudinal control of F-16 aircraft is reported in Section~\ref{sec:4}.  
Finally, concluding remarks are made in Section~\ref{sec:5}.

\section{Preliminaries}\label{sec:1.5}


A CG is an add-on algorithm to a  nominal closed-loop system (Plant + Controller) represented by a discrete-time model (system of difference equations),
\begin{equation}\label{equ:dynamics}
\bo{x}(t+1)=f(\bo{x}(t),\bo{v}(t)),
\end{equation}
where $\bo{x}(t) \in \mathbb{R}^n$ is the state vector aggregating the states of both the Plant and Controller, $\bo{v}(t) \in \mathbb{R}^m$ is the set-point command / reference vector, 
$m,n \in \mathbb{Z}_{>0}$ are positive integers, and $t \in \mathbb{Z}_{ \geq 0} $ is a non-negative integer which designates the discrete time instant.  

The CG enforces pointwise-in-time constraints expressed as 
\begin{equation}\label{equ:cnr}
(\bo{x}(t),\bo{v}(t)) \in \mathcal{C}~~\mbox{for all $t \in \mathbb{Z}_{\geq 0}$},
\end{equation}
where $\mathcal{C} \subset \mathbb{R}^{n+m}$ is a specified constraint set.  
This is done by monitoring and modifying the original reference command (set-point) $\bo{r}(t) \in \mathbb{R}^m$ to a safe reference command $\bo{v}(t) \in \mathbb{R}^m$,  see Figure~\ref{fig:basicRG}.
Note that as Eq. (\ref{equ:dynamics}) is a closed-loop system model,
Eq. (\ref{equ:cnr}) can represent both state and control constraints for the Plant \cite{garone2017reference}.

\begin{figure}[!ht]
	\begin{center}
		\includegraphics[width=12cm]{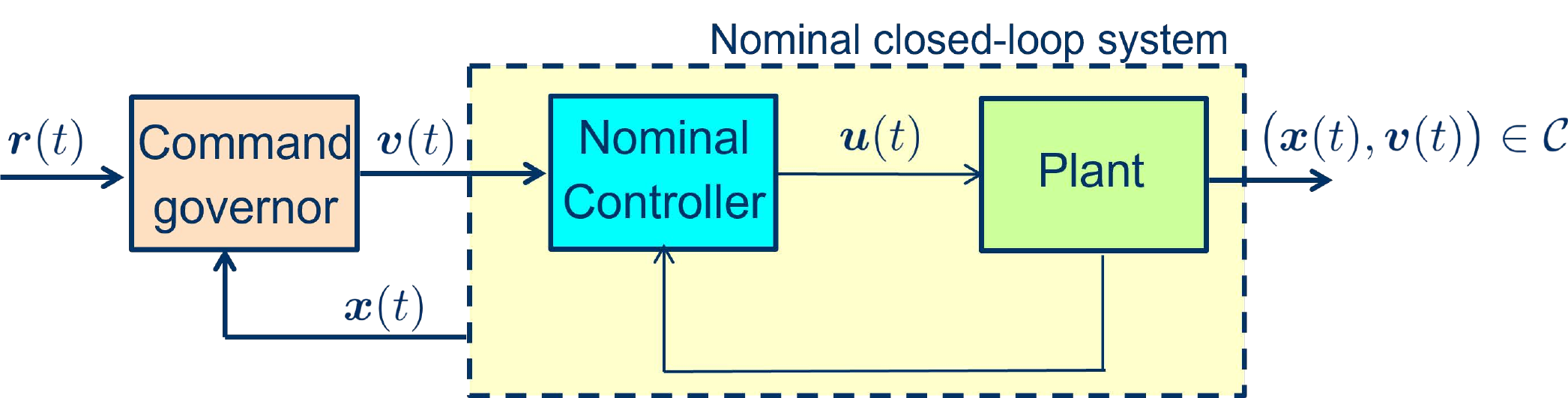}
		\caption{Command governor augmenting a nominal closed-loop system consisting of a Plant and a nominal Controller.}
		\label{fig:basicRG}
	\end{center}
\end{figure}

The MOAS, typically denoted as $O_\infty,$ is the set of all initial states and constant reference commands for which the subsequent response satisfies the constraints for all future times \cite{gilbert1991linear},
$$O_\infty=\{(\bo{x}(0),\bo{v}):~ \big(\bo{x}(t;\bo{x}(0),\bo{v} \big) \in \mathcal{C}, \, \forall t \in \mathbb{Z}_{\geq 0} \},$$
where $\bo{x}(t;\bo{x}(0),\bo{v})=A^t \bo{x}(0) + \sum_{j=0}^{t-1}A^j B v$ denotes the state trajectory of the system represented by Eq.~(\ref{equ:dynamics}) resulting from the initial state $\bo{x}(0)$ and the application of a constant reference command,
$\bo{v}(t)=\bo{v} $ for all $ t \in \mathbb{Z}_{\geq 0}$.
Any subset $P$ of $O_\infty$ is constraint admissible (satisfies constraints); such a subset is referred to as
``invariant for a fixed $\bo{v}$'' (or simply as ``invariant'')
if $(\bo{x},\bo{v}) \in P$ implies $(f(\bo{x},\bo{v}),\bo{v}) \in P$.

The Scalar Reference Governor (SRG) is the simplest RG algorithm, which searches for the closest admissible reference along the line segment connecting $\bo{v}(t-1)$ and $\bo{r}(t)$ by solving at each time instant the optimization problem:

\begin{maxi}|l|
 	{}{\kappa~~~~~~~~~~~~~~~~~~~~~~~~~~~~~~~~~~~~~~~~~~~~~~~~}{}{} \label{srg1}
	\addConstraint{0 \leq \kappa \leq 1}{}{}
	\addConstraint{\big( \bo{x}(t), \bo{v}(t-1)+\kappa ( \bo{r}(t)-\bo{v}(t-1) ) \big) \in P}{}{}
\end{maxi}
At each time instant $t$,  $\bo{v}(t)$ is assigned as the optimal $\bo{v}^*,$ i.e. $\bo{v}(t)=\bo{v}^*=\bo{v}(t-1)+\kappa^* (\bo{r}(t)-\bo{v}(t-1)),$ where $\kappa^*$ is the optimal solution to the optimization problem (\ref{srg1}).

The first SRG proposed in the literature for discrete-time systems considered  
Eq. (\ref{equ:dynamics}) as a linear system and $P$ in the optimization problem (\ref{srg1}) was chosen equal to $\tilde{O}_\infty$ -- a finitely-determined, invariant, constraint-admissible, inner approximation of $O_\infty$.
In later versions of the SRG, the requirement for $P$ in the optimization problem (\ref{srg1}) to be 
$\tilde{O}_\infty$ or even
invariant  was removed \cite{gilbert1999fast}, allowing a much greater freedom in selecting $P \subseteq \tilde{O}_\infty$.  Note that in this case, the optimization problem~(\ref{srg1}) is not guaranteed to be recursively feasible. However, under reasonable assumptions, it is possible to prove that if no feasible solution to (\ref{srg1}) exists and the reference is held constant, $\bo{v}(t)=\bo{v}(t-1),$ then the optimization problem \eqref{srg1} will be feasible again after a finite number of steps.
This property is particularly useful when,
to reduce the computational time and memory,
sets $P$ that are simple subsets of $\tilde{O}_\infty$ are used. Systematic procedures to generate
simpler $P$ from $\tilde{O}_\infty$ by removing almost redundant inequalities from its description and applying a pull-in transformation have been proposed  \cite{gilbert1999fast}.

One of the main strengths of SRG is its capability to manage constraints while being very computationally efficient. This is mainly due to the fact that in the SRG the selection of $\bo{v}(t)$ is reduced to the selection of the single scalar variable $\kappa \in [0,\, 1]$ which can be performed very efficiently (often in  closed form) for many types of constraints, see e.g.  \cite{nicotra2016fast}. 
Unfortunately, in the case $m>1$, this comes at the price of potentially slowing down the system response as certain   
$v_i(t)$ may be able to converge to the corresponding $r_i(t)$, $i=1,\cdots,m,$ faster than others and thus $\bo{v}(t)$ could be made closer to $\bo{r}(t)$ if not constrained to the line segment between $\bo{v}(t-1)$ and $\bo{r}(t).$


The above limitation can be overcome making use of the CG instead of SRG. Unlike SRG, the CG has more flexibility in choosing the reference $v(t)$  as
the solution to the following optimization problem:
\begin{argmini}|l|
	{\bo{v}}{\|\bo{r}(t)-\bo{v}\|^2_Q}{}{\bo{v^*}=} 
	\addConstraint{(\bo{x}(t),\bo{v}) \in P}{}{} \label{cg1}
\end{argmini}
where $Q=Q^{\rm T} \succ 0$,  $
\|\bo{r}(t)-\bo{v}\|^2_Q=(\bo{r}(t)-\bo{v})^{\rm T}Q (\bo{r}(t)-\bo{v}),$ 
and $P \subseteq O_\infty$ is invariant. In this case, at each time instant $t,$ the applied command is $\bo{v}(t)=\bo{v^*}.$

The price to be paid for using a CG instead of an SRG is that the optimization problem to be solved is not anymore a simple single variable optimization problem as in the SRG case, and, typically, 
it must be solved making use of an iterative 
optimization algorithm. For this reason, in practice, the CG is used almost exclusively when  the set $P \subseteq O_\infty$ is convex in $v$ for any fixed $x.$ Furthermore, to ensure the correct behaviour of the CG scheme (e.g., recursive feasibility of the optimization problem (\ref{cg1})), the set $P$ must be invariant.

The primary objective of this paper is to propose a modification of the conventional CG
 (\ref{cg1}) which:
\begin{enumerate} 
	\item Allows the CG to use a non-invariant set $P \subseteq O_\infty$;
	\item Allows to solve the optimization problem (\ref{cg1}) inexactly while still ensuring constraints satisfaction and finite-time convergence.
\end{enumerate}
This modification significantly extends the applicability of the CG to practical problems where finding invariant sets may be problematic and where exact optimization may not be feasible due to unreliability of the optimizers or limited computing power.  


\section{CG Convergence}\label{sec:2}

The Conventional CG convergence theory makes use of the following assumptions:

\begin{itemize}
	\item [A1] $P$ is positively invariant for any fixed $\bo{v}$, i.e., $(\bo{x},\bo{v}) \in P$ implies $(f(\bo{x},\bo{v}),\bo{v}) \in P$;
	\item [A2] For each $\bo{v}$ there exists a unique equilibrium $\bo{x_v}$ associated to a constant reference $\bo{v}$ such that $f(\bo{x_v},\bo{v})=\bo{x_v}$ and 
	$\bo{x_v}$ is Lipschitz continuous with respect to $\bo{v}$.  It is further assumed that the sets $R_P=\{\bo{v}|(\bo{x_v},\bo{v}) \in P\}$ and $P_x=\{\bo{v}|(\bo{x},\bo{v}) \in P\}$ are closed  and convex 
	for all $\bo{x}$; 
	\item [A3] There exists a scalar $\varepsilon>0$ such that for any fixed $\bo{v} \in R_P$ the set $P_v=\{\bo{x} | (\bo{x},\bo{v})\in P \}$ contains a ball of radius $\varepsilon$ centered at $\bo{x_v}$;
{	\item [A4] 
              $P \subseteq O_\infty \subseteq \mathcal{C}$};
	\item[A5] If $\bo{v}(t)-\bo{v}(t-1) \to 0$ as $t \to \infty$ then the solutions of 
	Eq. (\ref{equ:dynamics}) satisfy 
	$\bo{x}(t) \to \bo{x}_{\bo{v}(t)}$ as $t \to \infty$.
\end{itemize}

Assumptions A1-A4 are reasonable and are typically
made in the study of reference and command governors.
The assumption A5 is also reasonable.  For instance, if the discrete-time dynamics are linear,
$\bo{x}(t+1)=A \bo{x}(t)+B \bo{v}(t)$, and $A$ is Schur (all eigenvalues are inside the unit disk of complex plane), then
$\bar{x}_{\bo{v}}=(I-A)^{-1}B\bo{v}$, 
$$\bo{x}(t+1)-\bo{x}_{\bo{v}(t+1)}
=A(\bo{x}(t)-\bo{x}_{\bo{v}(t)})+(I-A)^{-1}B (\bo{v}(t)-\bo{v}(t+1)),$$ 
and A5 holds. 
For general nonlinear systems this property is similar  to the  discrete-time incremental Input-to-State (ISS) property \cite{tran2016incremental}. 

Under these assumptions it is possible to prove the following properties:
\begin{theorem}\label{TheoremCG1}
	Let the applied reference $\bo{v}(t)$ be managed by a CG based on solving at each sampling time the optimization problem in Eq.~\eqref{cg1}
	and let A1-A5 hold. If a $\bo{v}(0)$ exists such that $(\bo{x}(0),\bo{v}(0)) \in P$ then: 
	\item[1.] Constraints are satisfied at all time instants, $t \geq 0$;
	\item[2.] If the desired reference is constant, $\bo{r}(t)=\bo{r_0}$ for $t \geq \hat{t}$, 
	and, moreover, $\bo{r_0} \in R_P,$ then the sequence of $\bo{v}(t)$ converges in finite time to $\bo{r_0}$;  
	\item[3.] If the desired reference is constant, $\bo{r}(t)=\bo{r_0}$ for $t \geq \hat{t}$,  and  
	$\bo{r_0} \notin R_P,$ then the sequence of $\bo{v}(t)$ converges in finite time to the best approximation of $\bo{r_0}$ in $R_P,$, i.e., to $\bo{r^*} = arg \min_{\bo{v}\in R_P} ||\bo{v}-\bo{r}||_Q^2$ \end{theorem}

Here we review several elements of the proof as they inform subsequent modifications to the CG:

\begin{proof}

	1) Because of assumption A1, $P$ is invariant for any fixed $\bo{v}$. Then $(\bo{x}(t),\bo{v}(t-1))$ is always a feasible solution for the optimization problem \eqref{cg1}. Constraints satisfaction at all time instants can be concluded from assumption A4. 

	2),3) The starting point is to note that, because of assumption A1 and since $\bo{v}(t-1)$ is feasible at time $t$, the function $V(t)=||\bo{v}(t)-\bo{r_0}||_Q^2$ for $t>\hat{t}$ is non-increasing and bounded from above and below. Consequently, $\lim_{t \rightarrow \infty} V(t)$ exists and is finite, which also implies that $$\lim_{t \rightarrow \infty} \big[ V(t)-V(t-1) \big]=0.$$ 
	Note that in the general case this 	does not imply that 
	$\lim_{t \rightarrow \infty} \bo{v}(t)$ exists.
	 However, because of the convexity and closeness of $P_{\bo{x}(t)}$ (assumption A2) it is possible to prove that 
	\begin{equation}\label{equ:key1}
	   V(t-1)\geq V(t)+||\bo{v}(t)-\bo{v}(t-1)||_Q^2,
	\end{equation} 
	which implies that $\lim_{t \rightarrow \infty} 
	\big[\bo{v}(t)-\bo{v}(t-1) \big]=0$,
	 and hence $\bo{x}(t) \to \bo{x}_{\bo{v}(t)}$ as $ t \to \infty$ (assumption A5).  
	 
	 The rest of the proof is completed using assumption A3, that ensures feasibility, i.e., that $(\bo{x}(t),\bo{v}) \in P$, of any $\bo{v}$ such that $\|\bo{v} - \bo{v}(t-1)\| < \delta$ where $\delta>0$ is sufficiently small. Hence the only possible value for $\lim_{t \rightarrow \infty} \bo{v}(t)$ is $\bo{r_0}$ in the case $\bo{r_0} \in R_P$ and 
	 $\lim_{t \rightarrow \infty} \bo{v}(t) = \bo{r^*}$ otherwise. Furthermore, these limits are reached in finite time.
	 
A key argument of the entire proof is inequality (\ref{equ:key1}). To prove this inequality, the first step is to note that 
		since $P$ is invariant, $\bo{v}(t-1)$ is a feasible solution to optimization problem~\eqref{cg1} while $\bo{v}(t)$ is the optimal solution.  Note that $P_{\bo{x}(t)}$ is closed and convex, $\bo{v^-}=\bo{v}(t-1) \in P_{\bo{x}(t)}$ and $\bo{v^*}=\bo{v}(t)$ is the minimizer of  $F(\bo{v})=\|\bo{r_0}-\bo{v}\|_Q^2$ over $ \bo{v}
		\in P_{\bo{x}(t)}$. Then the necessary condition for optimality of $\bo{v^*}$ implies that, $d_+F(\bo{v^*};\bo{v} - \bo{v^*}) =(\nabla F(\bo{v}))^{\rm T} (\bo{v}-\bo{v^*}) \geq 0$ for any $\bo{v} \in P_{\bo{x}(t)}$ where $d_+$ stands for the Gateaux differential (directional derivative). Thus
		$$d_+f(\bo{v^*};\bo{v^-} -  \bo{v^*})=-2 (\bo{r_0}-\bo{v^*})^{\rm T} Q (\bo{v^-} - \bo{v^*})  \geq 0,$$
		and hence,
		\begin{equation}\label{equ:gateaux} (\bo{r_0} - \bo{v(t)})^{\rm T} Q (\bo{v^-} - \bo{v}(t)) \leq 0.
		\end{equation}
		Transforming inequality (\ref{equ:gateaux}) as
		$$(\bo{r_0}-\bo{v}(t)-\bo{v^-} + \bo{v^-})^{\rm T} Q(\bo{v^-} - \bo{r_0} + \bo{r_0} - \bo{v}(t) ) \leq 0,$$
		expanding and applying inequality (\ref{equ:gateaux}) again, it follows that
        \begin{equation}\label{equ:key2}
        ||\bo{v^-} - \bo{r_0}||_Q^2 \geq ||\bo{v}(t) - \bo{r_0}||_Q^2 + ||\bo{v^-} - \bo{v}(t)||_Q^2,
        \end{equation}	
        which implies the inequality (\ref{equ:key1}). 
\end{proof}


{\bf Remark~1: } It is worth to remark that unlike the SRG, which requires at time zero the knowledge of a feasible $\bo{v}(0)$ to start the algorithm, the CG only requires that a feasible $\bo{v}(0)$ exists as the CG itself will be able compute it. This  not only simplifies the start-up relative to the SRG, but also means that the CG has some implicit reconfiguration capability in case of impulsive disturbances. Indeed, if an impulsive disturbance changes the state of the system in such a way that the previously applied reference $\bo{v}(t-1)$ is not feasible anymore (i.e. $(\bo{x}(t),\bo{v}(t-1))\notin P$), the CG is able (whenever possible) to reconfigure the reference $\bo{v}(t)$ in such a way that $(\bo{x}(t),\bo{v}(t))\in P$. For this reason, the CG has also been used in Fault Tolerant control schemes \cite{casavola2007fault}. However it must be mentioned that, depending on the application, erratic jumps in $\bo{v}(t)$ due to occasional infeasibility caused by model mismatch or impulsive disturbances may not be necessarily preferable to maintaining 
previously applied reference, provided it is not permanently stuck, so this property is not necessarily an advantage.

\section{Modified Command Governor}\label{sec:3}

The proof of Theorem~\ref{TheoremCG1} reveals that the convergence results follow from the condition (\ref{equ:key2}) which, in the conventional CG case, is ensured thanks to the invariance of $P$  and the assumption that the CG is able to compute the optimal solution of the optimization problem \eqref{cg1}.  
The key observation behind this note is that if the condition (\ref{equ:key2}) satisfied in some other way, the results of Theorem~\ref{TheoremCG1} still follow without assuming invariance of $P$ or relying on exact optimization.  

A simple way to ensure that the condition (\ref{equ:key2}) holds is to use the following  logic-based condition for accepting a sub-optimal solution of the optimization problem~\eqref{cg1}. 

{\textbf{Modified CG}} Let $\bo{v^\prime}$ be a possibly sub-optimal solution
of the optimization problem  \eqref{cg1}. $\bo{v^\prime}$ is \textit{accepted}, i.e. 
$\bo{v}(t)=\bo{v^\prime}$ if :
\begin{itemize}
    \item $\bo{v^\prime}$ is feasible,  i.e. $(\bo{x}(t),\bo{v^\prime}) \in P$
    \item $\bo{v^\prime}$ satisfies \begin{equation}\label{equ:cond2}
||\bo{v^\prime}-\bo{r}(t)||_Q^2 \leq  ||\bo{v}(t-1)-\bo{r}(t)||_Q^2-||\bo{v^\prime}-\bo{v}(t-1)||_Q^2.
\end{equation}
\end{itemize}
Otherwise, $\bo{v^\prime}$ is \textit{rejected} and the previous value of the reference is held, i.e. $\bo{v}(t)=\bo{v}(t-1)$.

The following theorem shows that under very mild conditions on the solver properties, this simple logic allows to retain all of the properties of the conventional CG even if the solution of the optimization problem is inexact and if $P$ is not invariant.

\begin{theorem}\label{TheoremCG2}
	Let the set $P$ satisfy assumptions A2-A4, and let assumption A5 also hold. 	Consider a system 
	where $\bo{v}(t)$ is managed accordingly to the 
	\textbf{modified CG}. Under the only condition that there exist two scalars $\varepsilon^\prime,\delta^\prime>0$ such that whenever  
	\begin{equation}\label{equ:9pre}
	\|\bo{x}(t)-\bo{x}_{\bo{v}(t-1)}\|<\varepsilon^\prime,
	\end{equation}
	then the sub-optimal solution of the optimization problem \eqref{cg1} provides a feasible solution $\bo{v}(t)$ such that
	\begin{equation}\label{inexactCond}
	||\bo{v}(t)-\bo{r^*}(t)||_Q^2 \leq \max\left\{0, ||\bo{v}(t-1)-\bo{r^*}(t)||_Q^2-{ (\delta^\prime)^2} \right\} , 
	\end{equation}
	where $\bo{r^*}(t)=\arg \min_{\bo{v}\in R_P} ||\bo{v}-\bo{r}(t)||_Q^2,$
	then if $\bo{v}(0)$ is such that $(\bo{x}(0),\bo{v}(0)) \in P$:
	\begin{itemize}
		\item[1.] Constraints are satisfied at all time instants $t \geq 0$;
		\item[2.] If the desired reference is constant, $\bo{r}(t)=\bo{r_0}$ for $t \geq \hat{t}$, 
		and, moreover, $\bo{r_0} \in R_P,$ then the sequence $\bo{v}(t)$ converges in finite time to $\bo{r_0}$; 
		\item[3.] If the desired reference is constant, $\bo{r}(t)=\bo{r_0}$ for $t \geq \hat{t}$,  and  
		$\bo{r_0} \notin R_P,$ then the sequence $\bo{v}(t)$ converges in finite time to the best approximation of $\bo{r_0}$ in $R_P,$ i.e. $\bo{r^*}=\arg \min_{\bo{v}\in R_P} ||\bo{v}-\bo{r_0}||_Q^2$. 
	\end{itemize}
\end{theorem}

\begin{proof}
The constraints are satisfied as the property $(\bo{x}(t),\bo{v}(t)) \in O_\infty$ is maintained  for all $t \in \mathbb{Z}_{ \geq  0}$
despite the fact that $(\bo{x}(t),\bo{v}(t)) \in P$ may not hold. 
The acceptance/rejection logic based on the condition (\ref{equ:cond2}) 
ensures
$\|\bo{v}(t)-\bo{r}(t)\|_Q^2 \leq \|\bo{v}(t)-\bo{r}(t)\|_Q^2-\|\bo{v}(t)-\bo{v}(t-1)\|_Q^2$ for all $t \in \mathbb{Z}_{>  0}$.
This, coupled with the condition (\ref{inexactCond}),   ensures that properties 2 and 3 follow by the same arguments as in the proof of Theorem~1.
\end{proof}

According to Theorem~2 the only condition that a sub-optimal solver must satisfy in order to ensure the correct behaviour of the CG is that if $||\bo{x}(t)-\bo{x}_{\bo{v}(t-1)}||<\varepsilon^\prime$ then
the inequality \eqref{inexactCond} holds.  This condition is very reasonable in the CG setting. In fact, because of assumption A3 and of the Lipschitz continuity of $\bo{x}_{\bo{v}}$ w.r.t. to $\bo{v}$ (assumption A2) then for any $\varepsilon^\prime <\varepsilon$ there exists a $\delta^{\prime \prime}$ so that, whenever $||\bo{x}(t)-\bo{x}_{\bo{v}(t-1)}||<\varepsilon^\prime,$ then any $\bo{v}\in R_P$ such that $||\bo{v}-\bo{v}(t-1)|| \leq \delta^{\prime \prime}$ is feasible and thus $\bo{v}$ can be either moved  in the direction of  $\bo{r^*}(t)$ by a distance of $\delta^{\prime \prime}$ or be set equal to $\bo{r^*}(t)$.
This, in turn, guarantees the existence of a $\delta^\prime$ ensuring the inequality \eqref{inexactCond}.

This observation also allows to build the following algorithm that ensures the correct behaviour of the CG when integrated with an arbitrary optimization solver.

\begin{algorithm}[H]
\SetAlgoLined
\setstretch{0.65}
	\KwData{$\bo{x}(t)$, $\bo{r}(t)$, $\bo{v}(t-1)$}
	\KwResult{$\bo{v}(t)$}
	Compute $\bo{r}^*(t)=\arg \min_{\bo{v}\in R_P} ||\bo{v}-\bo{r}(t)||_Q^2$ \;
	Compute an approximate solution $\bo{v^\prime}$ of
	$$ \begin{array} {lcl}
	\bo{v^*} &=& \arg \min\limits_{\bo{v}} ||\bo{r^*}(t)-\bo{v}||_Q^2 \\
	& & \text{subject}\,\,\, \text{to} \\
	& & (\bo{x}(t),\bo{v}) \in P.
	\end{array}$$ 
	\eIf{$(\bo{x}(t),\bo{v^\prime}) \in P$ ~\mbox{AND}~ $\!|\!|\bo{v^\prime}\!-\!\!\bo{r}(t)|\!|_Q^2 \!\! \leq \!\! |\!|\bo{v}(t\!-\!1)\!-\!\bo{r}(t)|\!|_Q^2\!\!-\!|\!|
	\bo{v^\prime}\!-\!\bo{v}(t\!-\!1)|\!|_Q^2\!$}{
		set $\bo{v^{\prime \prime}}=\bo{v^\prime}$ \;
	}{
		set $\bo{v^{\prime \prime}} = \bo{v}(t-1)$ \;
	}
	\eIf{$||\bo{x}(t)-\bo{x}_{\bo{v}(t-1)}|| \leq \varepsilon^\prime$ ~\mbox{AND}~ $||\bo{v^{\prime \prime}}-\bo{v}(t-1)||_Q < \delta^\prime$}
   	{return $\bo{v}(t)= \bo{v}(t-1)+min\left\{1,\frac{\delta^{\prime \prime }}{||\bo{r^*}(t)-\bo{v}(t-1)||}\right\}{\color{red} }(\bo{r^*}(t)-\bo{v}(t-1))$}{return $\bo{v}(t)=\bo{v^{\prime \prime}}$\;}
	\caption{Command Governor for non-exact solver}
	\label{AlgoCG6}
\end{algorithm}
%

Note that if $Q=I$ then $\delta^\prime = \delta^{\prime \prime} $.
Note also that an SRG can be viewed as a special case of inexact CG in which $Q=I,$ $\bo{r}(t)\in R_P,$ and a search over the line segment between $\bo{v}(t-1)\in R_P$ is used as an inexact solution.  Note also that in the SRG approach the inequality (\ref{inexactCond}) is satisfied whenever the condition (\ref{equ:9pre}) holds.

{\bf Remark 2:}
The requirement of inequality (\ref{inexactCond}) holding whenever the condition (\ref{equ:9pre}) holds can be relaxed. For instance, it is sufficient that there exists $N \in \mathbb{Z}_{>0}$ such that the inequality (\ref{inexactCond})  holds at least once in every sequence of length $N$ of consecutive time steps $t$ for which condition (\ref{equ:9pre}) holds.

{\bf Remark~3: } Note that this modified CG algorithm loses the ``reconfiguration'' capabilities mentioned in Remark~1. In fact, whenever $(\bo{x}(t),\bo{v}(t-1))\notin P$ we are implicitly assuming that by keeping the command constant, the constraints are always satisfied.

{\bf Remark~4: } Depending on the shape $P$, in order to reduce the number of discarded $\bo{v}(t)$ as a result of violation of the condition (\ref{equ:cond2}), the following optimization problem can be used in place of the optimization problem \eqref{cg1}:
\begin{mini}|l|
	{\bo{v}}{\|\bo{r}(t)-\bo{v}\|_Q^2+\|\bo{v}(t-1)-\bo{v}\|_Q^2}{}{}, \label{cg1mod}
	\addConstraint{(\bo{x}(t),\bo{v}) \in P}{}{}.
\end{mini}
This optimization problem is also a QP.

\section{F-16 Aircraft Longitudinal Flight Control Example}\label{sec:4}

In this section we consider an example of longitudinal control of F-16 aircraft based on the continuous-time aircraft model presented in  \cite{sobel1985design}. This model represents linearized closed-loop aircraft dynamics at an altitude of $3000$ ft and
$0.6$ Mach number in straight and level subsonic flight.
The model has been converted to discrete-time assuming a sampling period of $5$ msec.  
The resulting discrete-time model has 
the form of Eq. (\ref{equ:dynamics}) with
$$f(\bo{x},\bo{v})=A\bo{x}+B\bo{v},$$
where
$$
{\small
A= { \left(\begin{array}{ccccc} 0.9998 & 3.126\times 10^{-5} & 0.006366 & 0.0008041 & 0.001198\\ -0.01104 & 0.9928 & 0.1892 & -0.07997 & -0.00731\\ 0.0002201 & 0.004952 & 0.9941 & -0.001009 & -0.001217\\ 0.3035 & 0.0844 & 0.6711 & 0.8547 & -0.007991\\ -0.5769 & -0.08625 & -0.953 & 0.04102 & 0.9148 \end{array}\right),~ }}
{\small
 B=\left(\begin{array}{cc} 5.314 \times 10^{-6} & 0.0002335\\ 0.01105 & -2.445 \times 10^{-5}\\ 1.334 \times 10{-5} & -0.0002335\\ -0.2676 & -0.03565\\ 0.1873 & 0.3896 \end{array}\right).}
$$
The components  of the state vector $x \in  \mathbb{R}^5$ are: the flight path angle (deg),
the  pitch rate (deg/sec), the angle of attack (deg), the  elevator deflection (deg), and the  flaperon deflection (deg), respectively. The components of the reference command vector $\bo{v} \in \mathbb{R}^2$ components are: the commanded pitch angle (deg), and commanded flight path angle (deg), respectively.  
Upper and lower bound constraints  are prescribed on the 
elevator deflection,
flaperon deflection,
  elevator deflection rate, 
  flaperon deflection rate,
  and angle of attack. 
These constraints can be written as
\begin{equation}\label{equ:cnr10}
(\bo{x},\bo{v}) \in \mathcal{C}=\{(\bo{x},\bo{v}):~ \bo{y_c} = C_c \bo{x} + D_c \bo{v} \in Y_c\},
\end{equation}
with $$C_c=\left(\begin{array}{ccccc} 0 & 0 & 0 & 1.0 & 0\\ 0 & 0 & 0 & 0 & 1.0\\ 65.0 & 17.82 & 142.3 & -30.5 & -1.68\\ -122.0 & -17.95 & -200.6 & 8.412 & -17.89\\ 0 & 0 & 1.0 & 0 & 0 \end{array}\right),
D_c=\left(\begin{array}{cc} 0 & 0\\ 0 & 0\\ -57.6 & -7.34\\ 40.4 & 81.6\\ 0 & 0 \end{array}\right)$$
and $$Y_c=\{\bo{y_c}:~ \Gamma \bo{y_c} - \bo{\gamma} \leq \bo{0}\}$$
$$=
[-25,25] \times [-20, 20] \times  
[-42, 42] \times [-56, 56] \times [-4, 4],$$
where the set $Y_c$ is the Cartesian product of the intervals restricting the range of each of the components of $y_c$ in Eq.~(\ref{equ:cnr10}), which is a $5 \times 1$ vector and its compoments have units of deg, deg, deg/sec, deg/sec and deg, respectively.  The matrices $\Gamma$ and $\bo{\gamma}$ are
$10 \times 5$ and $10 \times 1$, respectively.
The limits on the elevator and flaperon deflection and deflection rates are based on
\cite{sobel1985design}.  The angle of attack limit of $\pm 4$ deg has been made tighter than usual to create a more challenging  scenario for the CG.  In practice, tight limits on angle of attack could be imposed by the flight envelope protection system when  flying in presence of significant wind gusts or high turbulence \cite{richardson2013envelopes} (especially near a trim condition already at a high angle of attack), during aerial refueling, when flying in tandem with a drone or if wing icing has occurred.

In this example we compare three CG implementations:  The first implementation (conventional CG) is based on optimization problem (\ref{cg1mod}) with $P=\tilde{O}_\infty$ defined by $748$ inequalities and the primal-dual active set algorithm 
{\tt qpkwik} implemented in Matlab was used to solve the QP problem (\ref{cg1mod}) with the maximum number of iterations limited to $200$.  This solver was chosen as it appears to be one of the fastest options for solving QP problems for optimization problems that, like the CG, have a small number of optimization variables and large number of constraints.
The second implementation (modified CG) was the same as the first except for the maximum number of iterations limited to $3$. The QP solver was warm-started in both implementations 1 and 2. The third implementation (also modified CG) was based on $P$ with $106$ inequalities obtained from $\tilde{O}_\infty$ by systematic elimination of the almost redundant inequalities and a pull-in procedure \cite{gilbert1999fast}.  This reduction in the number of inequalities translates into more than a $7$-fold reduction in ROM size needed to store $P$. In this third implementation, we use as an approximate solver based on a modified scalar reference governor
update (\ref{srg1}) that assumes the following form 
\begin{equation}\label{equ:modsrg}
\bo{v} = \bo{v}(t-1) +\kappa \bo{E}(t) (\bo{r}(t) - \bo{v}(t-1)), 	
\end{equation}
where $\bo{E}(t)$ is alternating between 
$$\left\{ \left[\begin{array}{c} 1 \\ 0  \end{array} \right], \left[\begin{array}{c} 0 \\ 1  \end{array} \right], 
\left[\begin{array}{c} 1 \\ 1  \end{array} \right]
\right\}.$$
This strategy is motivated by the idea of applying time distributed coordinate descent. Since only a scalar parameter $\kappa$ is optimized, the minimizer can be easily found by evaluating an explicit expression \cite{gilbert1999fast}.  To guarantee convergence to constant constraint admissible inputs, every third update is made along the line segment connecting $\bo{v}(t-1)$ and $\bo{r}(t)$; this ensures that the relaxed condition in Remark 2 holds with $N=3$. 
In each of these three implementations, the logic based condition (\ref{equ:cond2}) is applied.  
The responses are shown in Figures~\ref{fig:sim11}-\ref{fig:sim14}.  
The time history of the maximum of constraints at each time instant, i.e., of
$\max \{\Gamma( C_c \bo{x}(t)+ D_c \bo{v}(t) )- \gamma \}$ is plotted in Figure~\ref{fig:sim13}.  
		As this maximum stays less or equal than zero  (whose value is designated by a dashed black line in Figure~\ref{fig:sim13}),   the constraints are satisfied by each of
Implementations.  Figure~\ref{fig:sim14} shows that the condition (\ref{equ:cond2}) is activated sparingly for Implementation 2 and frequently for Implementation 3.  
The computational time statistics were
tallied from $19$ simulations in sequence in Matlab running under
Windows 10 on Microsoft Surface 7 tablet for
each of the Implementations.  The computation
times (averaged over time instants and $19$ runs) were $9.539$ msec for Implementation 1, $4.978$ msec for 
Implementation 2 and $1.9087$ msec for Implementation 3.
As is clear from Figures~\ref{fig:sim11}-\ref{fig:sim12}, the response is slightly slower in case of Implementation 3 as compared to Implementations 1 and 2.
 
\begin{figure}[!ht]
	\begin{center}
		\includegraphics[width=5cm]{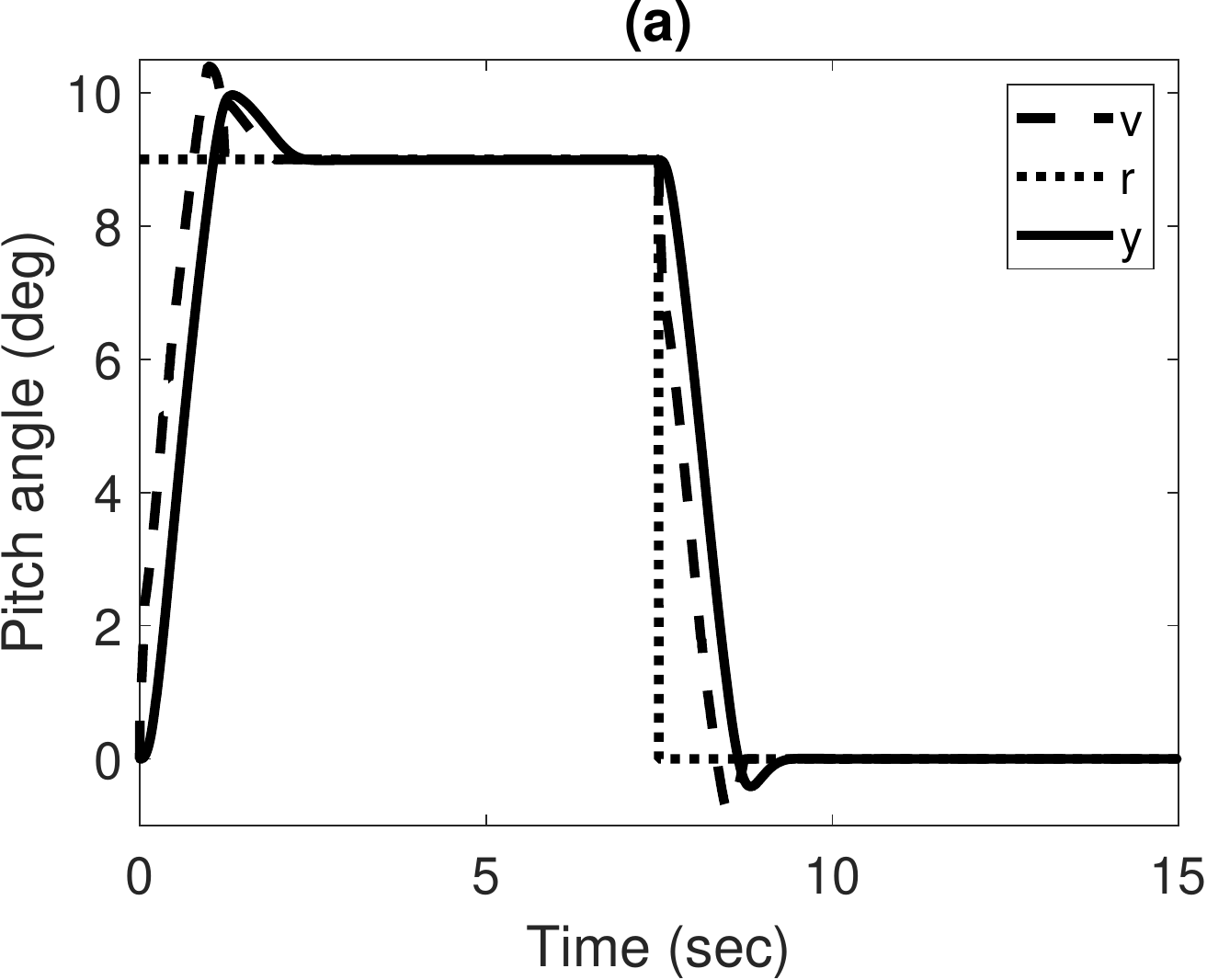}~~
		\includegraphics[width=5cm]{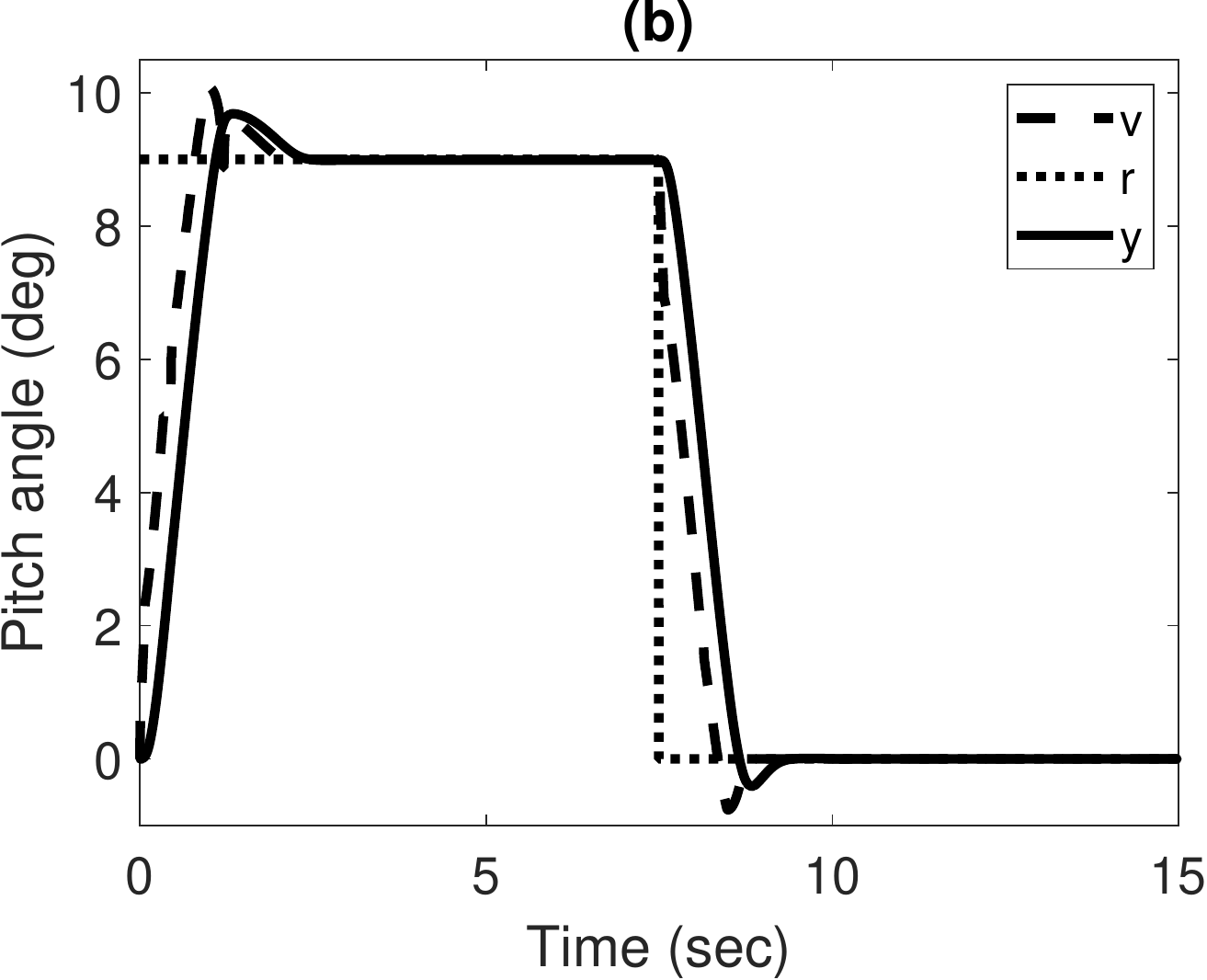}~~
		\includegraphics[width=5cm]{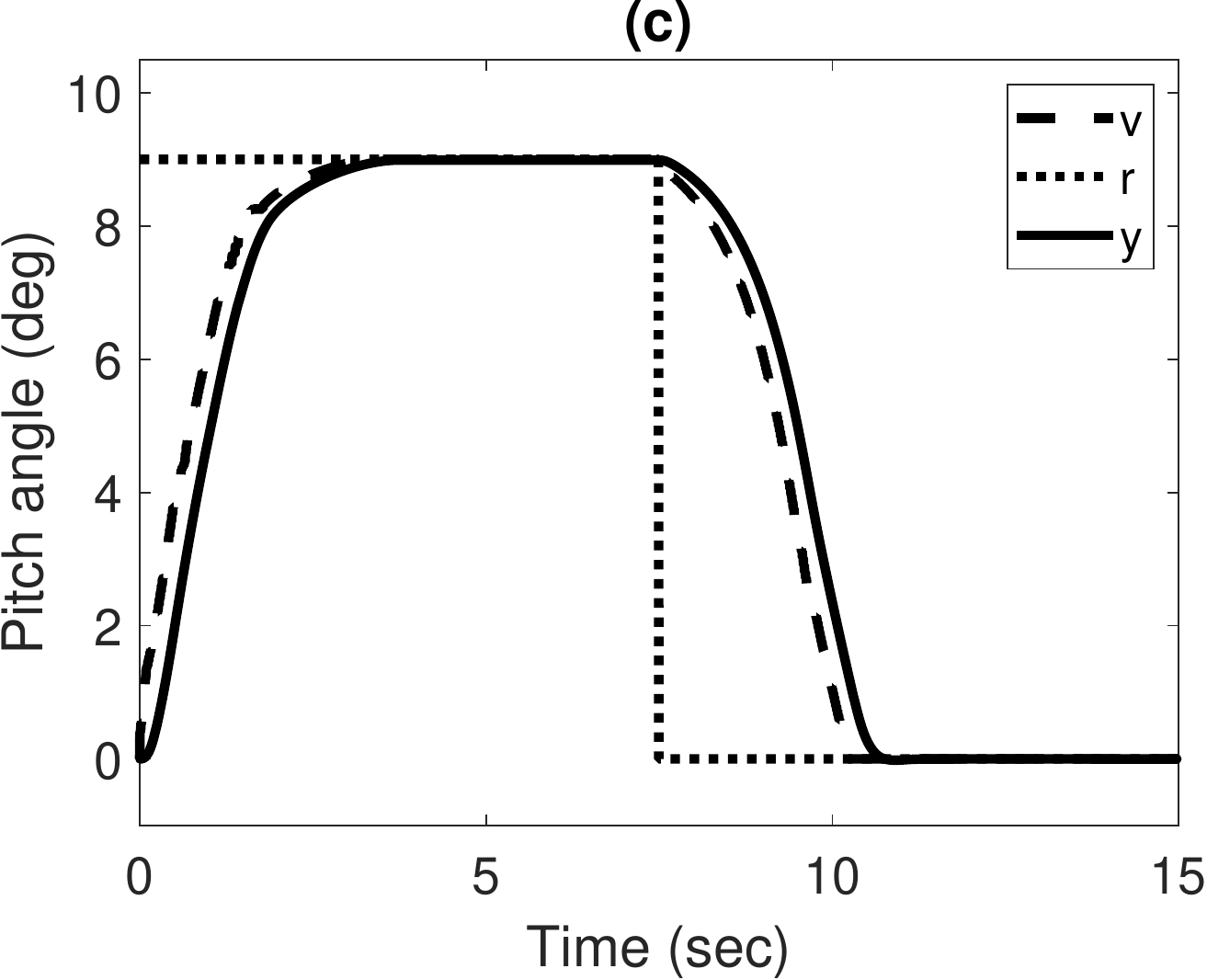}
		\caption{Time histories of pitch angle command, modified pitch angle command and actual pitch angle: Implementation 1 (a), Implementation 2 (b), Implementation 3 (c). }
		\label{fig:sim11}
	\end{center}
\end{figure}

\begin{figure}[!ht]
	\begin{center}
		\includegraphics[width=5cm]{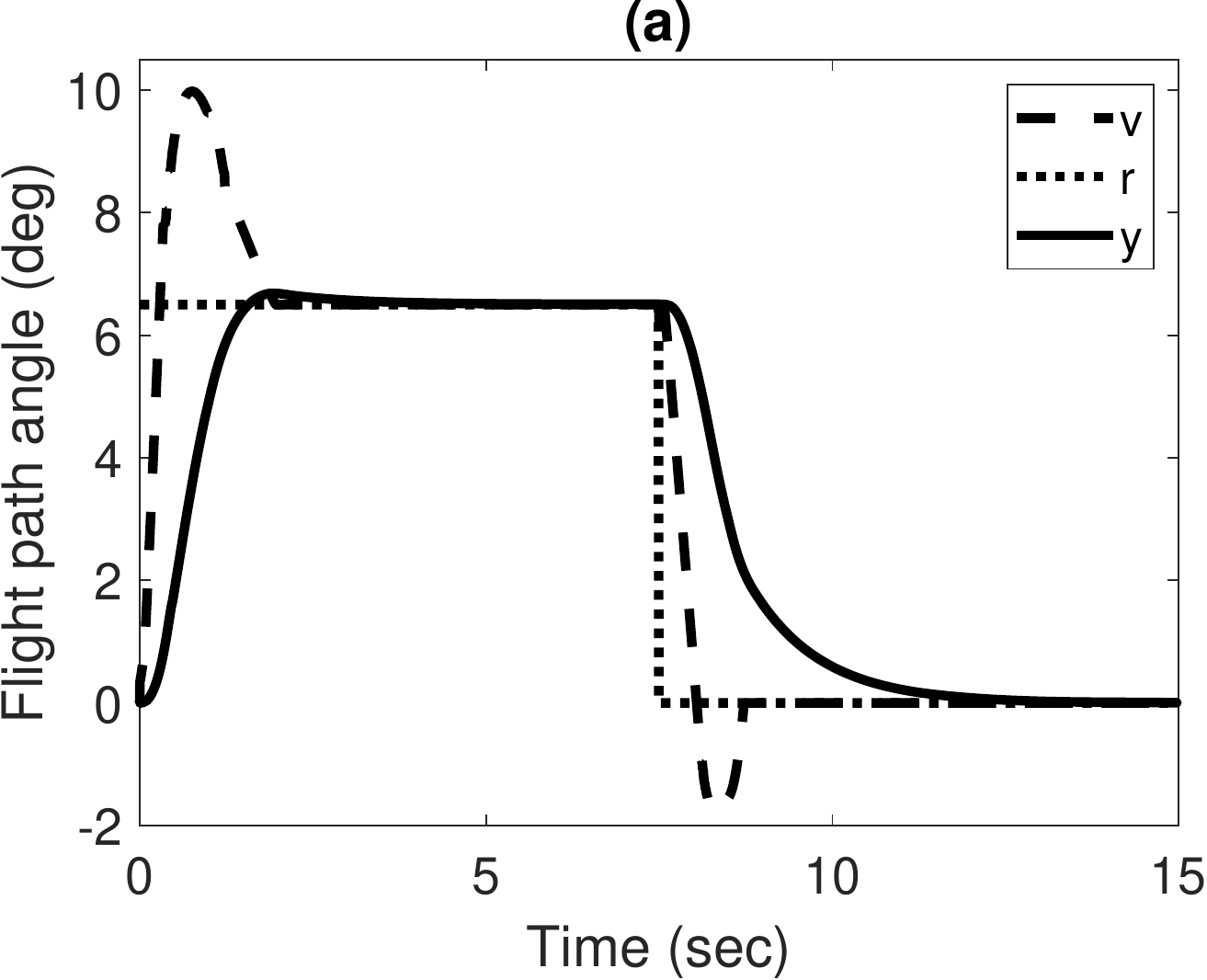}~~
		\includegraphics[width=5cm]{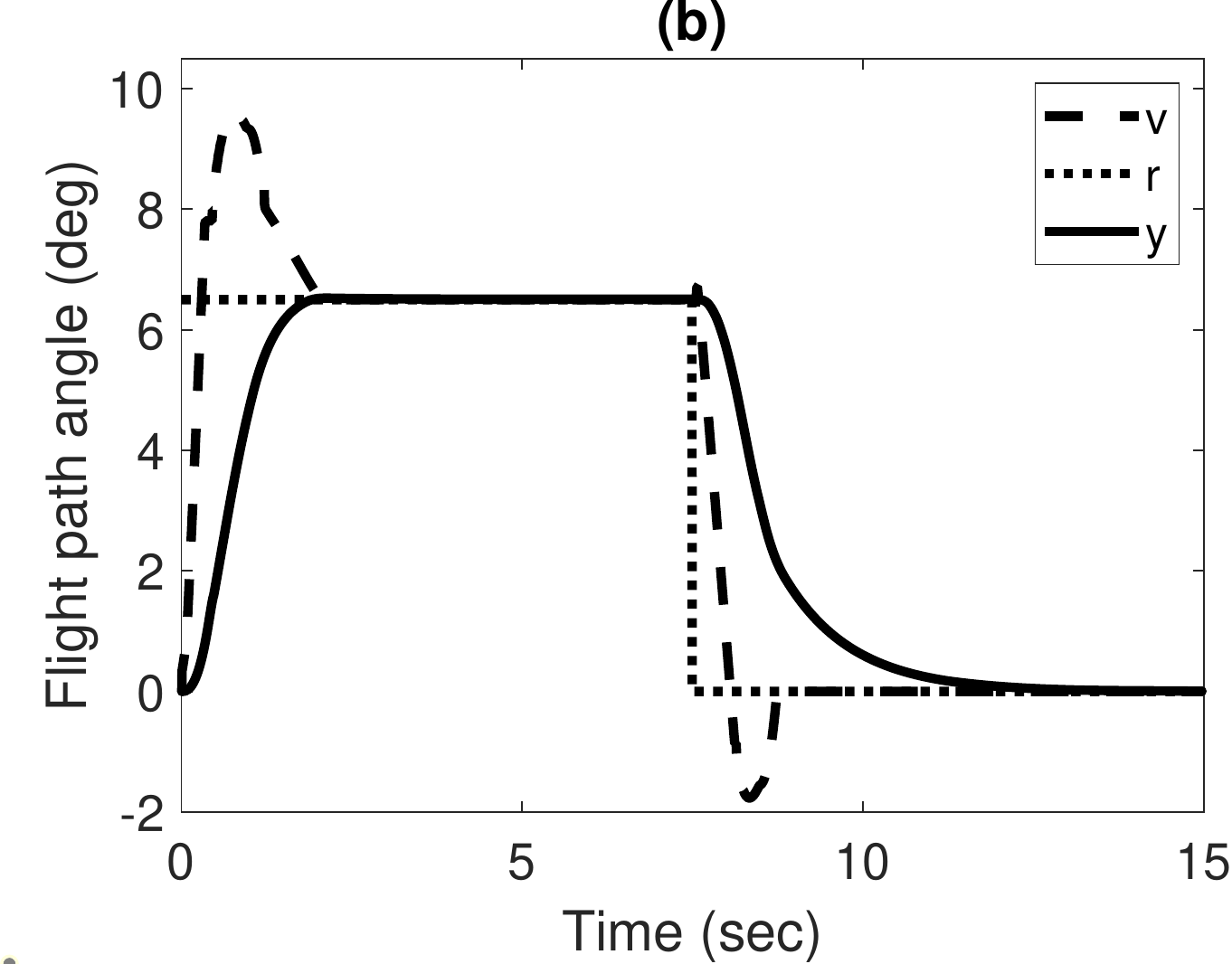}~~
		\includegraphics[width=5cm]{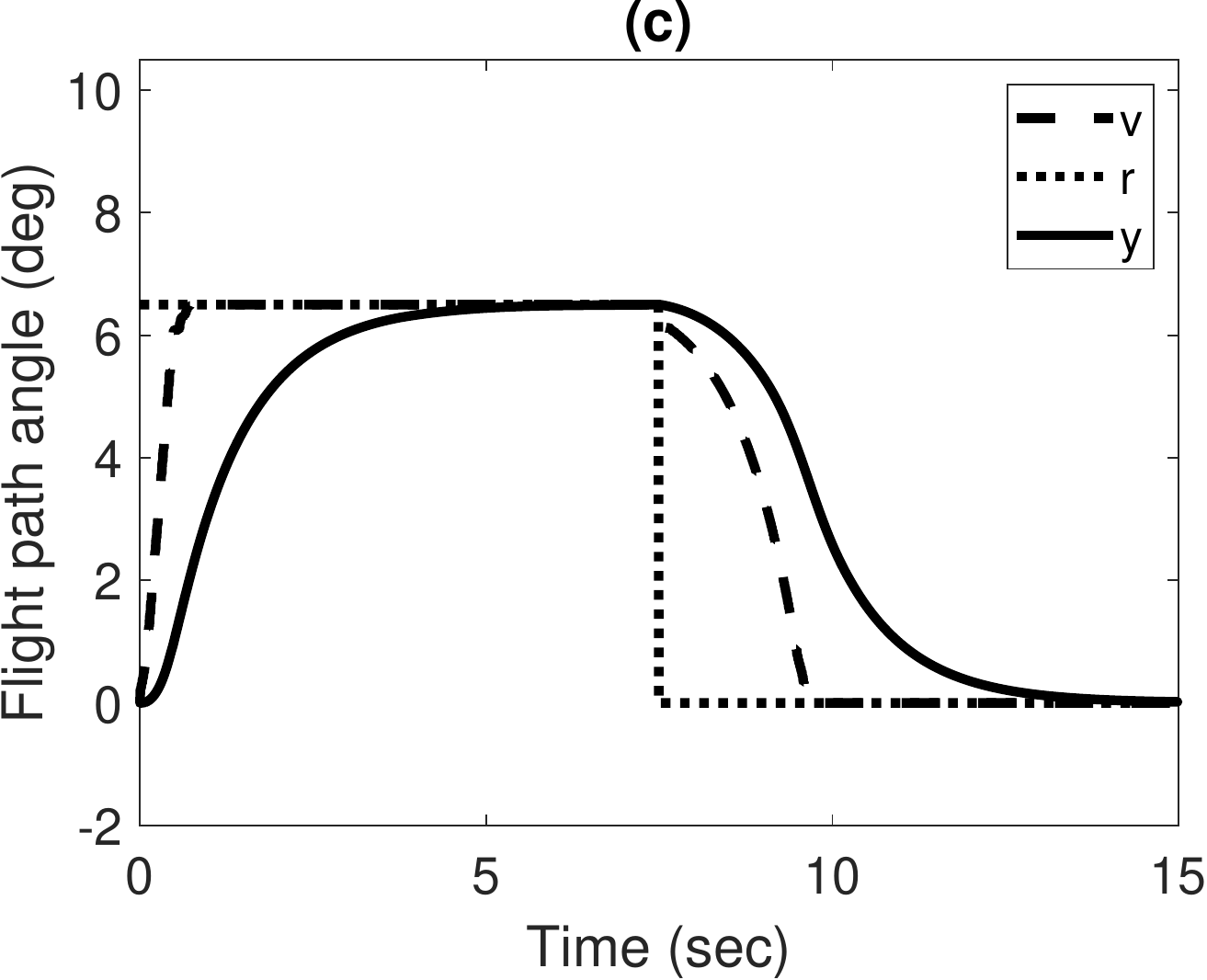}
		\caption{Time histories of flight path angle command, modified flight path angle command and actual flight path angle:
		Implementation 1 (a), Implementation 2 (b), Implementation 3 (c).
		}
		\label{fig:sim12}
	\end{center}
\end{figure}

\begin{figure}[!ht]
	\begin{center}
		\includegraphics[width=5cm]{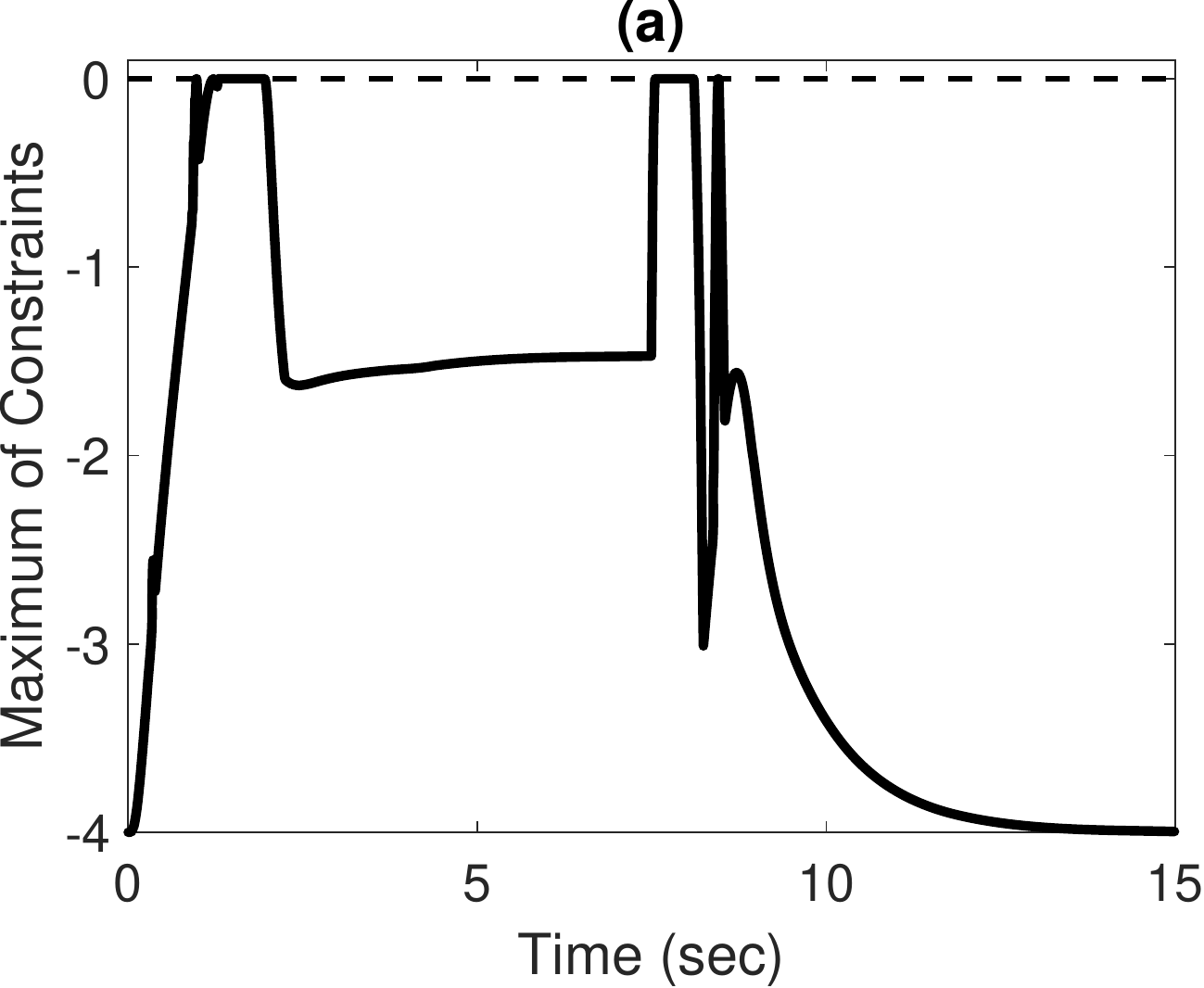}~~
		\includegraphics[width=5cm]{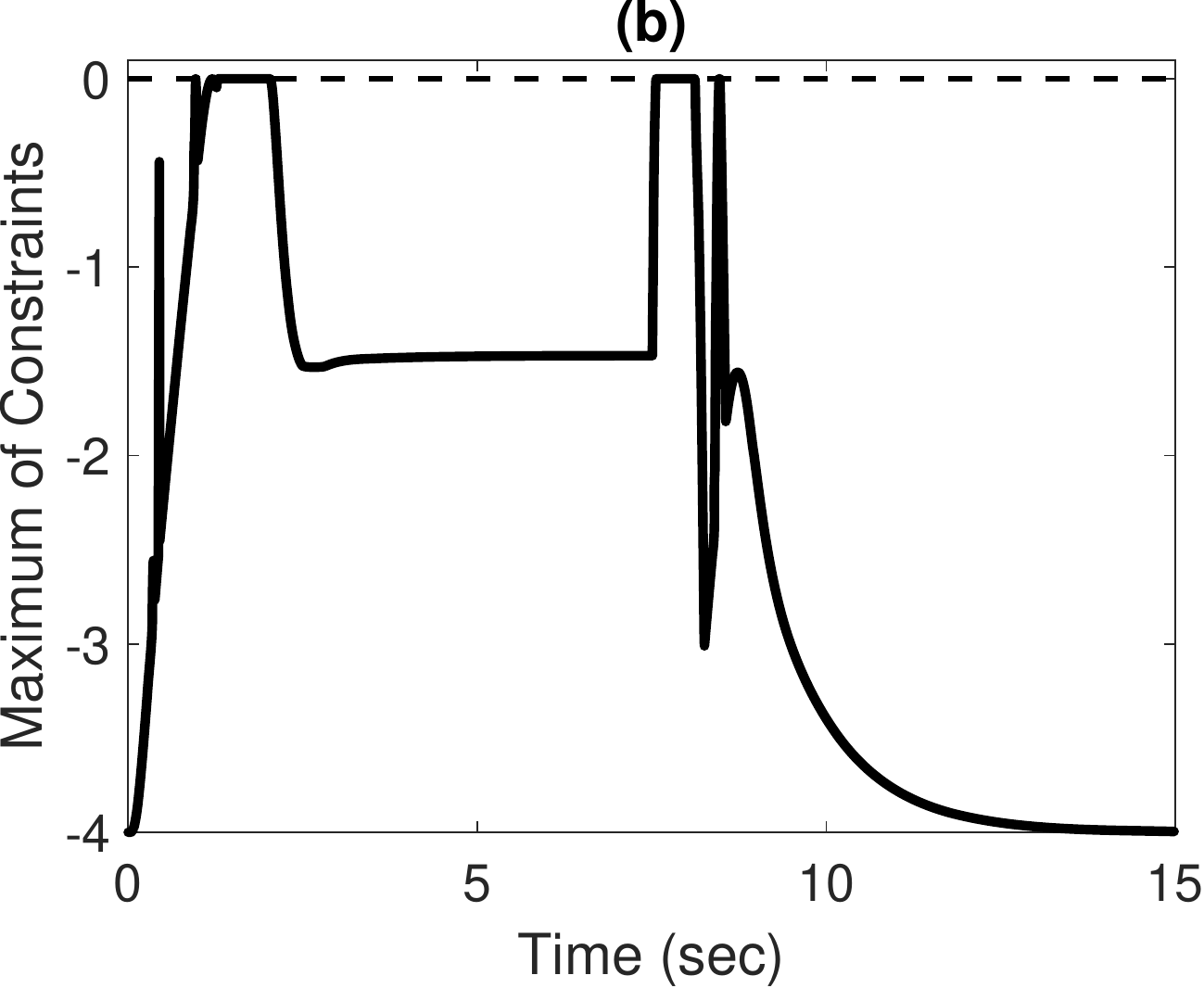}~~
		\includegraphics[width=5cm]{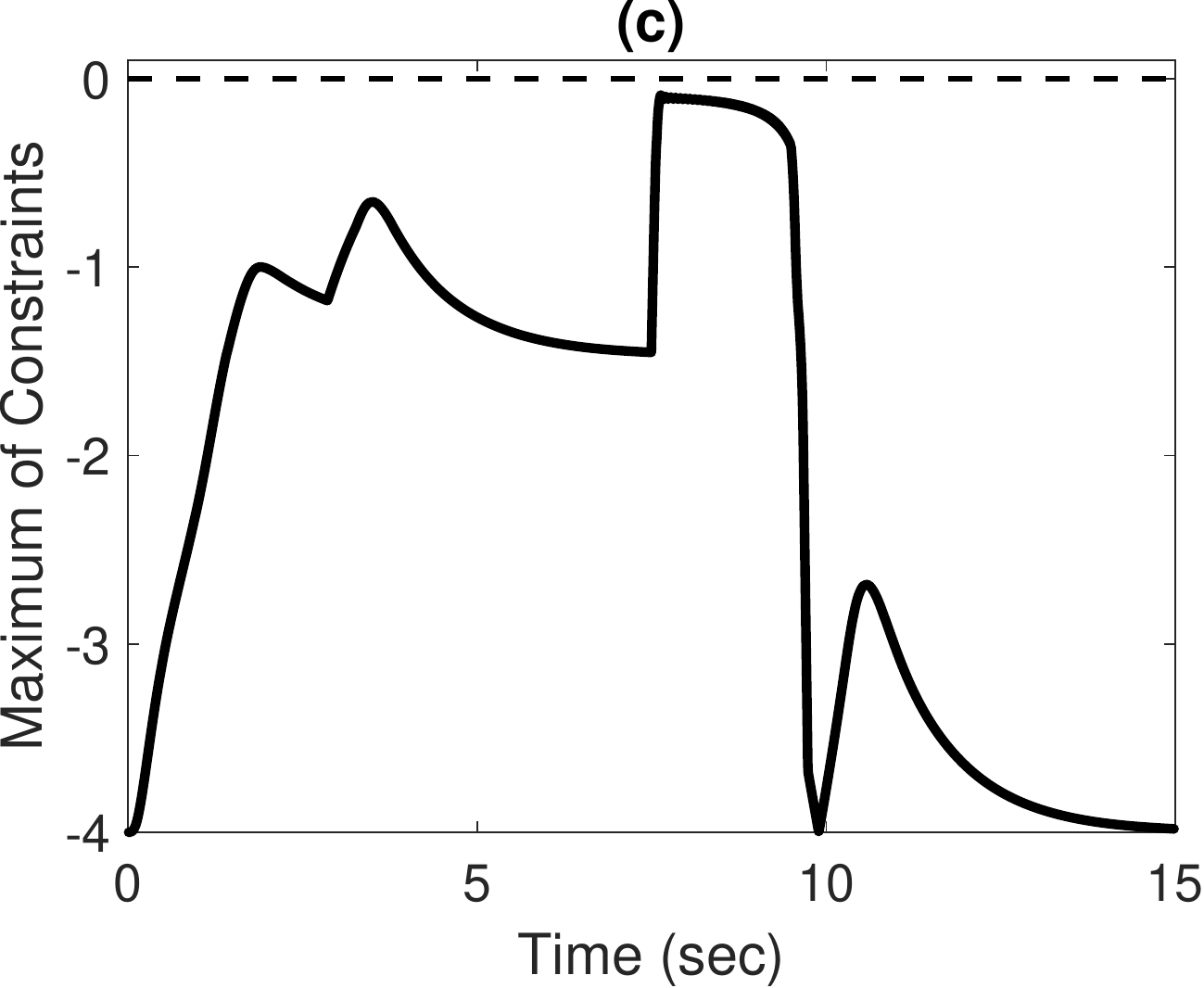}
		\caption{Time history of the maximum of constraints, $\max \{
		\Gamma_Y( C x(t)+ D v(t) )-
		\gamma_Y \}$: Implementation 1 (a), Implementation 2 (b), Implementation 3 (c).  }
		\label{fig:sim13}
	\end{center}
\end{figure}

\begin{figure}[!ht]
	\begin{center}
		\includegraphics[width=5cm]{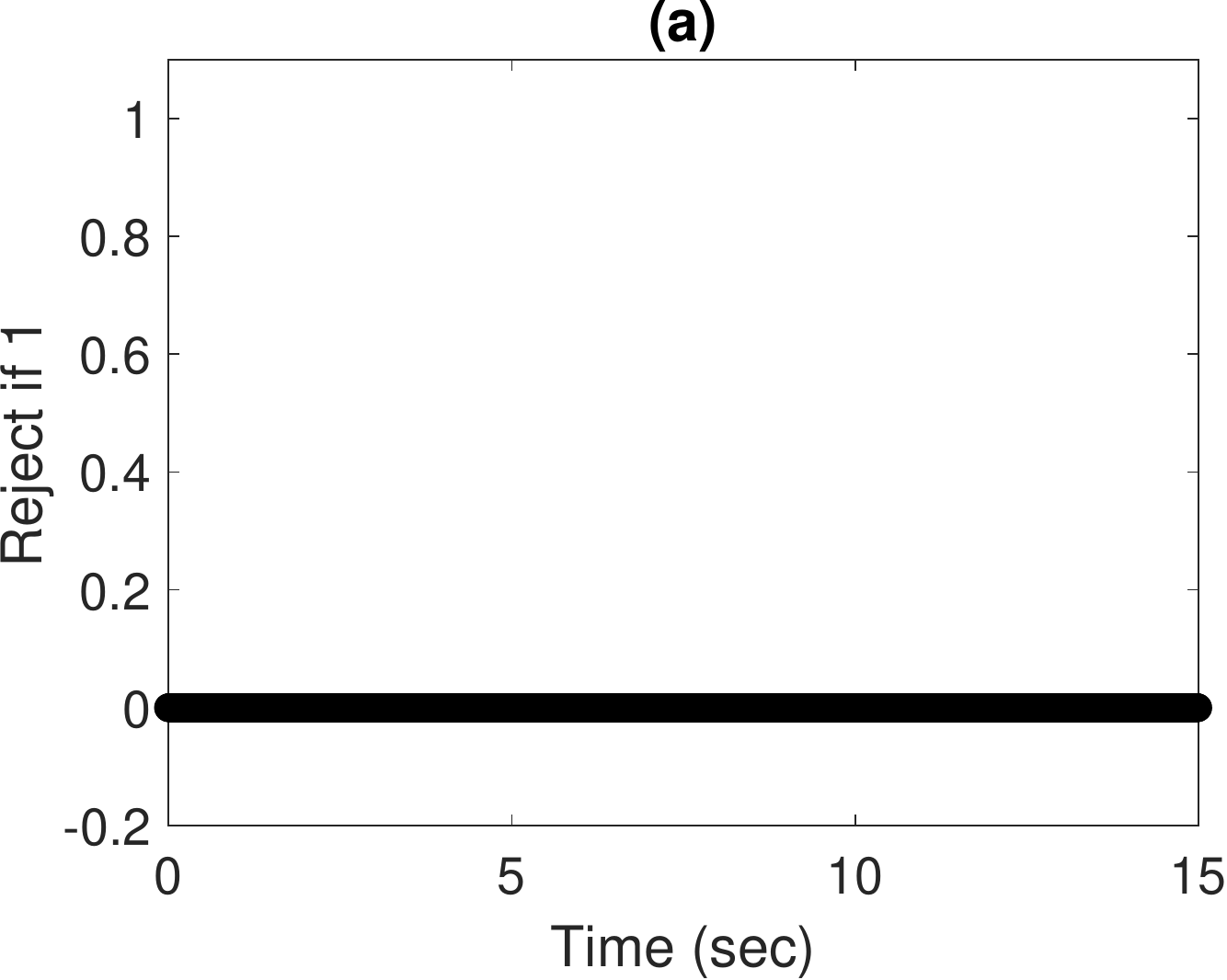}~~
		\includegraphics[width=5cm]{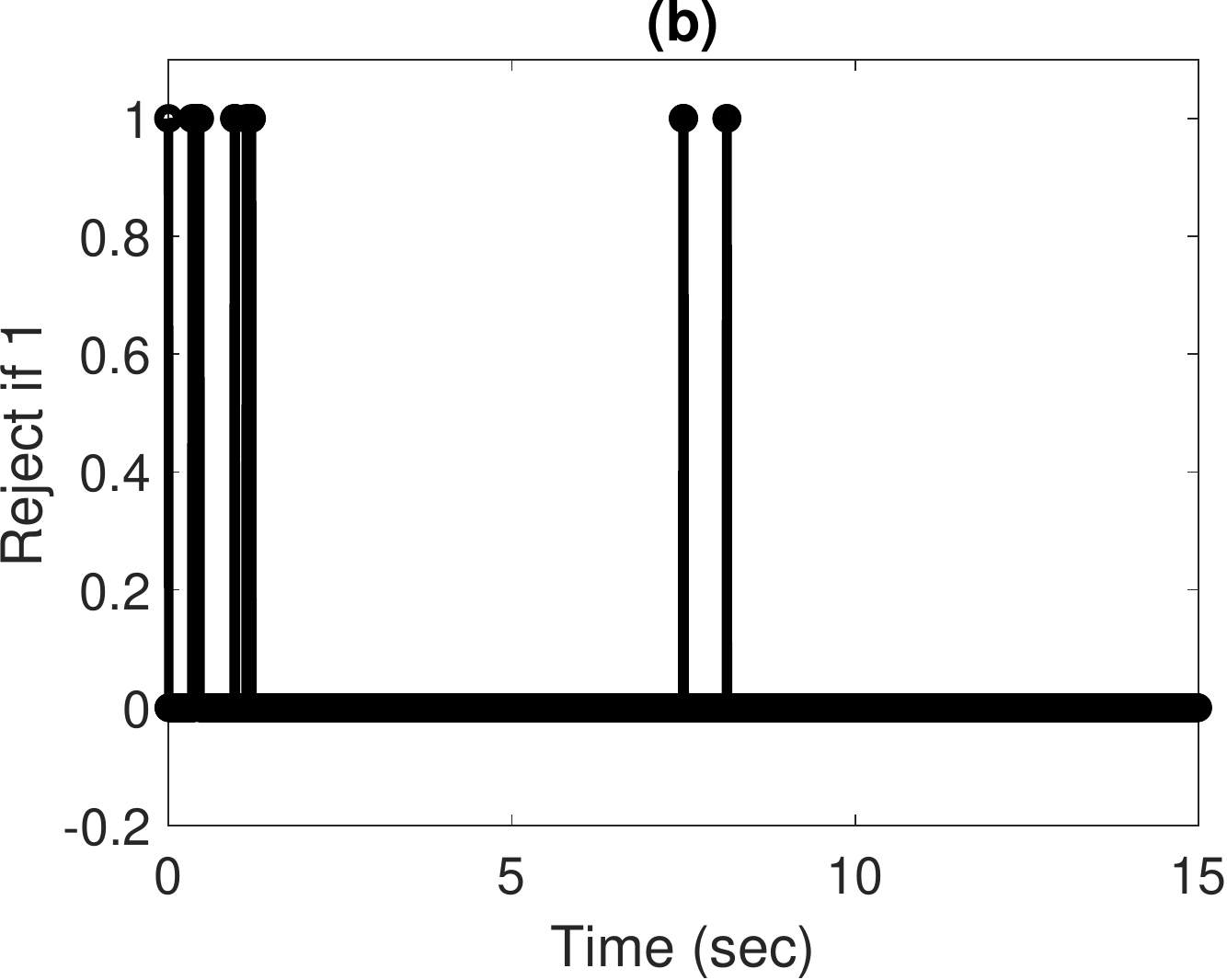}~~
		\includegraphics[width=5cm]{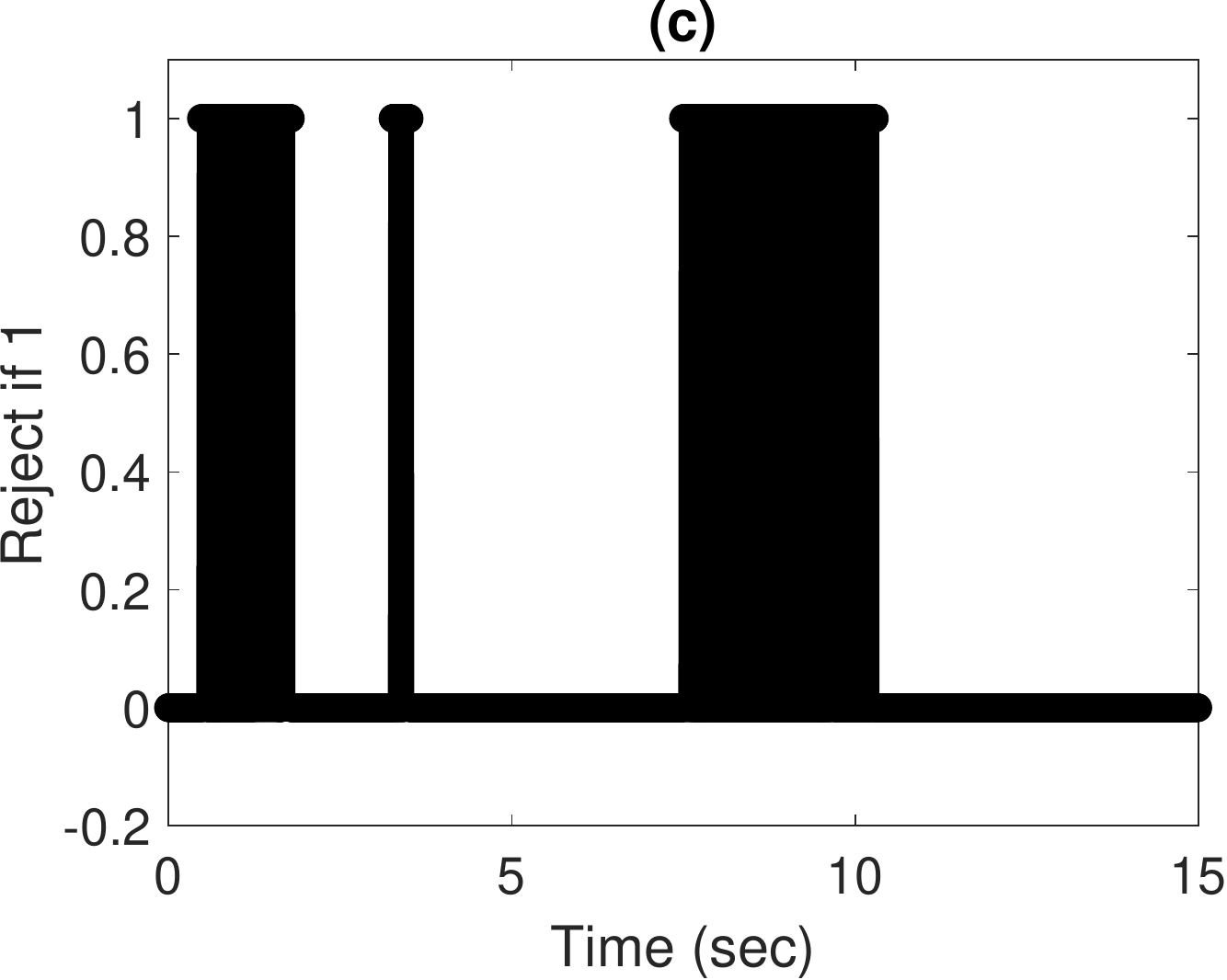}
		\caption{Time instants at which condition (\ref{equ:cond2}) is violated, $\bo{v}^\prime$ is rejected, and previous value of reference is held for Implementation 1 (a), Implementation 2 (b) and  Implementation 3 (c).  }
		\label{fig:sim14}
	\end{center}
\end{figure}

\section{Concluding Remarks}\label{sec:5}

The Command Governor (CG) is an add-on scheme to a nominal closed-loop system used to satisfy state and control constraints through reference command modification to maintain state-command pair in a safe set.  By modifying the CG logic, it is possible to implement CG without requiring this safe set to be invariant or relying on exact optimization, while retaining constant command finite time convergence properties typical of CG.  An F-16 longitudinal flight control simulation example has been reported that  demonstrated the potential of the proposed approach for significant reduction in the computation time and in ROM size required for implementation.

\section*{Acknowledgement} The authors would like to thank Dr. Dominic Liao-McPherson for
the code of primal-dual active set solver {\tt qpkwik} used in the numerical experiments.  The second author acknowledges  the support of
AFOSR under the grant number FA9550-20-1-0385 to the University
of Michigan.

\bibliography{sample}

\end{document}